\documentclass[11pt,a4paper]{article}

\usepackage{mathrsfs}
\usepackage[all]{xy}
\usepackage[latin1]{inputenc}
\usepackage[english]{babel} 
\usepackage{amsthm}
\usepackage{amssymb}
\usepackage{amsfonts, mathrsfs}
\usepackage{amsmath}
\usepackage[numbers, sort&compress]{natbib}
\usepackage{graphicx}
\usepackage{vmargin, enumerate}
\usepackage{setspace}
\usepackage{paralist}
\usepackage[pdftex]{hyperref}
\DeclareMathOperator{\Hom}{Hom}

\begin{document}
\newcommand{\z}{\mathbb{Z}}
\newcommand{\zn}{\mathbb{Z}_{n}}
\newcommand{\zp}{\mathbb{Z}_{p}}
\newcommand{\pf}{\noindent {\bf Proof:} }
\newcommand{\rmk}{\noindent {\bf Remark:} }
\newcommand{\gp}{\Gamma^{+}}
\newcommand{\gm}{\Gamma^{-}}
\newcommand{\gpp}{\Gamma^{++}}
\newcommand{\gmm}{\Gamma^{--}}
\newcommand{\ep}{\varepsilon}
\newcommand{\mac}{\mathcal{C}}
\newcommand{\bmac}{\bar{\mathcal{C}}}
\newcommand{\la}{\lambda}
\newcommand{\ra}{\rightarrow}

\newtheorem{claim}{Claim}
\newtheorem{theorem}{Theorem}[section]
\newtheorem{proposition}[theorem]{Proposition}
\newtheorem{problem}[theorem]{Problem}
\newtheorem{conjecture}[theorem]{Conjecture}
\newtheorem*{definition}{Definition}
\newtheorem*{definitions}{Definitions}
\newtheorem{example}[theorem]{Example}
\newtheorem{corollary}[theorem]{Corollary}
\newtheorem{lemma}[theorem]{Lemma}
\newtheorem*{remark}{Remark}
\newtheorem*{remarks}{Remarks}

\title{Quasi-random oriented graphs}
\author{Simon Griffiths}
\date{}

\author{Simon Griffiths\, \\[3mm]  $\,$\footnotesize{IMPA, Est. Dona Castorina 110, Jardim Bot\^anico, Rio de Janeiro, Brazil}}

{\renewcommand{\thefootnote}{\relax} \footnotetext{The author is supported by CNPq (Proc. 500016/2010-2),
e-mail\,:  \texttt{sgriff-at-impa.br}}}

\maketitle

\vspace{-0.6cm}

\abstract{\small{We show that a number of conditions on oriented graphs, all of which are satisfied with high probability by randomly oriented graphs, are equivalent.  These equivalences are similar to those given by Chung, Graham and Wilson \cite{CGW} in the case of unoriented graphs, and by Chung and Graham \cite{CG} in the case of tournaments.  Indeed, our main theorem extends to the case of a general underlying graph $G$ the main result of \cite{CG} which corresponds to the case that $G$ is complete.

One interesting aspect of these results is that exactly two of the four orientations of a four-cycle can be used for a quasi-randomness condition, i.e., if the number of appearances they make in $D$ is close to the expected number in a random orientation of the same underlying graph, then the same is true for every small oriented graph $H$.}}

\section{Introduction}

The concept of quasi-randomness (also known as pseudo-randomness) of graphs has been well-studied since it was introduced by Thomason \cite{T1},\cite{T2}.  Much work since has focused on demonstrating that various concepts of quasi-randomness are equivalent.  To this aim, Chung, Graham and Wilson \cite{CGW} listed no less than seven equivalent quasi-randomness conditions for graphs.  These included a condition on the number of four-cycles in the graph, conditions on the number of edges between large subsets and a condition on the eigenvalues of the graph.  We call these conditions \emph{quasi-randomness} conditions as, in all these cases, they ask that a given graph parameter (eg., number of four-cycles), is close to the value this parameter takes (with high probability) in a random graph.

A little later Chung and Graham \cite{CG} continued this line of research by proving similar results for tournaments (oriented graphs obtained by orienting a complete graph).  Again the conditions considered ask that a given parameter of the tournament be close to its expected value in a random orientation.  One way to think of the results of Chung and Graham is that the underlying graph has been fixed ($G=K_{n}$) and we are concerned only with the properties of the orientation.  In this article we show that similar results may be proved relative to any fixed underlying graph $G$.

With this concept of quasi-randomness a random orientation $D$ of a complete bipartite graph (for example) is very likely to have strong quasi-randomness properties.  One might argue that such an oriented graph should not be considered quasi-random as it does not contain close to the correct number of all small oriented graphs (for example, oriented triangles).  Presumably someone making this argument would prefer a stronger concept of quasi-randomness that did guarantee a close to correct number of all small oriented graphs.   Stronger quasi-randomness conditions of this type may be easily obtained by combining one of our quasi-randomness conditions with a quasi-randomness condition of Chung, Graham and Wilson \cite{CGW} on the underlying graph.\footnote{In the case that the density is not equal to $1/2$ then refer instead to the version of their result given by Simonovits and S\' os \cite{SS}.}  (For example, if the number of homomorphic copies of Type II four cycles (see diagram on page \pageref{diag}) in an oriented graph $D$ is close to the number of copies in a random orientation of $G(n,1/2)$ (i.e., $\approx n^{4}/32$) then the same is true for all small oriented graphs).\vspace{0.2cm}

We remark that in addition to the extension to tournaments \cite{CG}, Chung and Graham also studied quasi-randomness conditions for set systems \cite{CG2} and hypergraphs \cite{CG3}.

There have been a number of articles in recent years establishing the equivalence of other quasi-randomness conditions for graphs, including \cite{HL}, \cite{S}, \cite{SY}, \cite{SS}, \cite{Y}.  We also mention that Kalyanasundaram and Shapira \cite{KS} have recently established another quasi-randomness condition for tournaments: for every fixed even integer $k \ge 4$, if close to half of the $k$-cycles in a tournament $T$ are even (have an even number of anti-cyclic edges), then $T$ must be quasi-random.  Their result confirms a conjecture of Chung and Graham \cite{CG}, and (as they note) may be extended to setting described here.  For background on quasi-randomness we refer the reader to the survey of Krivelevich and Sudakov \cite{KS} and the informative discussion of quasi-randomness given by Gowers in the introduction of \cite{G}.\vspace{0.2cm}

We now prepare to state the main theorem.  While the statement will not be at all surprising to those familiar with previous work on quasi-randomness we nonetheless require quite a lot of notation.  We introduce some notation immediately, so that Theorem \ref{equiv} may, at least informally, be understood; notation will be given in full in Section \ref{notation} and we encourage the reader to look ahead if any clarification is required.

We begin by considering the four possible orientation of the four cycle.  Consider the four-cycle $\{12,23,34,41\}$.  There are 16 possible orientations of this four-cycle and they fall into four equivalence classes. 

\vspace{0.3cm}
\noindent
\begin{tabular}{l c p{0.5cm} c p{0.5cm}  c p{0.5cm} c} \label{diag}

& 

\xymatrix{ \bullet \ar[r] & \bullet \ar[d] \\  \bullet \ar[u] &  \bullet \ar[l] } 

& &

\xymatrix{ \bullet \ar[r] & \bullet \ar[d] \\  \bullet \ar[u]\ar[r]  &  \bullet } 

& &

\xymatrix{ \bullet \ar[r] & \bullet \\  \bullet \ar[u]\ar[r]  &  \bullet \ar[u] } 

& &

\xymatrix{  \bullet & \bullet \ar[l]\ar[d] \\  \bullet \ar[u]\ar[r]  &  \bullet  } 

\vspace{0.3cm}
\\  

&  Type I & & Type II & & Type III & & Type IV 

\vspace{0.3cm}
\\

Number  & 2 & & 8 & & 4 & & 2 
\vspace{0.2cm}
\\

Proportion  & $\frac{1}{8}$ & & $\frac{1}{2}$ & & $\frac{1}{4}$ & & $\frac{1}{8}$

\end{tabular}\vspace{0.7cm}

The above table includes a count of the number of ways each type can occur as an orientation of our four-cycle, and, dividing by 16, the proportion of orientations of our four-cycle which result in a cycle of this type.  These proportions then also describe the proportion of each type one would expect to find if one were to orient the edges of a graph at random.  Thus, these proportions are important to our discussion of quasi-randomness.  We count the number of homomorphic copies of cycles of a given type by summing the number of homomorphic copies of each of the oriented four cycles of that type. We write $\hom_{I}(D), \hom_{II}(D), \hom_{III}(D), \hom_{IV}(D)$ for the corresponding counts and $\hom(C_{4},D)$ for the number of homomorphic copies of an unoriented four-cycle -- equivalently the total number of homomorphic copies of the 16 orientations -- so that \begin{equation}\label{homup} \hom(C_{4},D)=\hom_{I}(D)+\hom_{II}(D)+\hom_{III}(D)+\hom_{IV}(D)\, . \end{equation}

We denote by $M$ the adjacency matrix of an oriented graph $D$ and by $\la_1,\dots ,\la_n$ the eigenvalues of $M$ in order of decreasing absolute value.

We shall consider the class of \emph{partially oriented graphs}, i.e., graphs in which some subset of the edges are oriented.  In a partially oriented graph $D$ we denote by $e(D)$ the number of edges of $D$ and by $\vec{e}(D)$ the number of oriented edges.  Let $\mathscr{O}_{k}$ denote the set of oriented graphs on $k$ vertices, and $\mathscr{P}_{k}$ the set of partially oriented graphs on $k$ vertices.  We denote by $\bar{D}$ the underlying graph of a partially oriented graph $D$.  We now state the main theorem.

\begin{theorem}\label{equiv} The following conditions on an oriented graph $D$ are equivalent, in the sense that any one of them, with any positive value of its parameter ($\alpha,\beta,\delta,$...) can be deduced from any other, when the parameter of the latter is taken sufficiently small.  When the same parameter appears in both then the implied equality holds.\vspace{0.4cm}

\noindent
\begin{displaymath}
\begin{array}{l l l} 
\textup{P1} &\quad & \Big| \hom_{IV}(D) - \frac{1}{8}\hom(C_{4},D)\Big| \le \alpha n^{4}\vspace{0.4cm}
\\
\textup{P2} & & \Big| \hom_{II}(D)-\frac{1}{2}\hom(C_{4},D)\Big| \le \frac{\delta}{2} n^{4}\vspace{0.4cm}
\\
\textup{P3} & & \Big| \hom(H,D)- 2^{-e(H)}\hom(\bar{H},D) \Big |\le \beta n^{k} \qquad\text{ for all $H\in \mathscr{O}_{k}$} \quad (\forall k\in\mathbb{N}) \vspace{0.4cm}
\\
\textup{P4} & & \Big| \hom(H,D)- 2^{-\vec{e}(H)}\hom(\bar{H},D) \Big |\le \eta n^{k} \qquad\text{ for all $H\in \mathscr{P}_{k}$} \quad (\forall k\in\mathbb{N}) \vspace{0.4cm}
\\
\textup{P5} & & \sum_{x,x',y,y'\in V} M_{xy}M_{xy'}M_{x'y}M_{x'y'} \le \delta n^{4} \vspace{0.4cm}
\\
\textup{P6} & & \vec{e}(A,B) - \vec{e}(B,A)\le \gamma n^{2} \qquad \qquad\qquad\qquad\text{ for all } A,B\subset V \vspace{0.4cm}
\\
\textup{P7} & & bias_{1-\ep}(D)\le \ep n^{2}\vspace{0.4cm}
\\
\textup{P8} & & \sum_{i=1}^{n}\la_{i}^{4}\le \delta n^{4} \vspace{0.4cm}
\\
\textup{P9} & & |\la_{1}|\le \zeta n\vspace{0.1cm}
\end{array}
\end{displaymath}\end{theorem}\vspace{0.2cm}

We remind the reader that any notation not already defined will be given in Section \ref{notation}.\vspace{0.2cm}

We remark that our quasi-randomness conditions (P1 -- P9) are not in exact correspondence with those given by Chung and Graham for the case of tournaments \cite{CG}.  For example Condition 11 of Chung and Graham states that: there exists some ordering $\pi$ [of the vertices of $D$] so that [the graph consisting of edges ascending in that ordering] is quasi-random [in the sense of Chung, Graham and Wilson].  To make the equivalent statement in our setting we would be required to define a concept of quasi-randomness for two-colourings of a graph.  This should not be difficult to do and should be largely analogous to the present discussion of quasi-random orientations of graphs, but we choose not to investigate it further here.

The proof of Theorem~\ref{equiv} is very similar in spirit, and often in detail, to the proofs given by Chung, Graham and Wilson~\cite{CGW} in the case of graphs and Chung and Graham \cite{CG} in the case of tournaments.  Indeed, many of our proofs are simply minor alterations of proofs given in these articles.  The author would also like to acknowledge the influence of the 2005 lecture course of Gowers on ``Quasirandomness''.  The proof of Theorem~\ref{equiv} is given in Section~\ref{secproof}.  We first introduce notation in Section \ref{notation} and establish some initial results concerning the various orientations of a four-cycle in Section~\ref{oofc}.

\section{Notation}\label{notation}

As mentioned above, we shall consider \emph{partially oriented graphs}, graphs in which some subset of the edges are oriented.  In a partially oriented graph $D$ both notations $xy\in E(D)$ and $\vec{xy}\in E(D)$ occur, the former being used in the case that the edge $xy$ (either oriented or unoriented) is present in $D$ and the latter indicating that this edge is oriented from $x$ to $y$.  Since both (unoriented) graphs and oriented graphs are contained within the category of partially oriented graphs we need only define notation for partial oriented graphs $D$.

Given a partially oriented graph $D$, we write $e(D)$ for the total number of edges of $D$, $\Gamma(x)$ the set of neighbours of a vertex $x$, and $d(x)$ the degree of $x$.  While we write $\vec{e}(D)$ for the number of arcs  (oriented edges) of $D$, $\Gamma^{+}(x)$ and $\Gamma^{-}(x)$ for the out- and in- neighbourhoods of $x$, and $d^{+}(x)$ and $d^{-}(x)$ for the cardinalities $|\Gamma^{+}(x)|$ and $|\Gamma^{-}(x)|$.  Recall that $\bar{D}$ denotes the underlying graph of $D$.

For sets $A,B\subseteq V(D)$ denote by $e(A,B)$ the number of edges between $A$ to $B$, and by $\vec{e}(A,B)$ the number of arcs from $A$ to $B$.

For $x,y\in V(D)$, and a pair $\sigma,\tau \in\{ +,-\}$, we define $\sigma\tau$ joint degree $d^{\sigma \tau}(x,y)=|\Gamma^{\sigma}(x)\cap \Gamma^{\tau}(y)|$.

Given two partially oriented graphs $D$ and $H$, a homomorphism from $H$ to $D$ is a function $f:V(H)\ra V(D)$ satisfying:\vspace{0.1cm}

\begin{tabular}{p{0.5cm} l}
(i) &  $\vec{f(x)f(y)}$ is an arc of $D$ for every arc $\vec{xy}$ of $H$. 
\\
(ii) &  $f(x)f(y)$ is an edge of $\bar{D}$ for every edge $xy$ of $H$.
\end{tabular}\vspace{0.1cm}

We denote by $\hom(H,D)$ the number of homomorphisms from $H$ to $D$, i.e., the number of homomorphic copies of $H$ in $D$.

The adjacency matrix $M$ of an oriented graph $D$ on $n$ vertices is the $n$ by $n$ matrix with rows and columns indexed by $V(D)$ and with entries described by

\begin{displaymath} 
M_{xy} = \left\{ \begin{array}{lll} 
1 & &\textrm{if $\vec{xy}\in E(D)$}\\ 
-1 & &\textrm{if $\vec{yx}\in E(D)$}\\ 
0 & &\textrm{otherwise} 
\end{array} \right. 
\end{displaymath} 

We denote by $\la_{1},\dots ,\la_{n}$ the eigenvalues of the adjacency matrix $M$ of $D$, in the order of decreasing absolute value.  Since $M$ is anti-symmetric and has real-valued entries, all its eigenvalues are imaginary.

For $\nu\in (0,1)$ and $A,B\subset V(D)$, say that $(A,B)$ is $\nu$-biased if $\vec{e}(B,A)\le \nu \vec{e}(A,B)$.  Define $bias_{\nu}(D)$ to be the maximum value of $\vec{e}(A,B)$ for a $\nu$-biased pair $(A,B)$, i.e., 
$$bias_{\nu}(D)=\max\big\{\vec{e}(A,B):\, A,B\subset V \text{ are such that } \vec{e}(B,A)\le \nu \vec{e}(A,B)\big\}\, .$$
These parameters of oriented graphs were studied recently in \cite{AGH}.

\section{Orientations of four-cycles}\label{oofc}

In this section we establish some very elementary inequalities that hold for all oriented graphs $D$.  The main result we establish states, in a sense, that cycles of Type IV can never be under-represented and cycles of Type II can never be over-represented in any oriented graph $D$.

\begin{proposition} \label{propIIIV} The following inequalities hold for all oriented graphs $D$: 

\begin{tabular}{l p{0.5cm} l}

(i) & & $\hom_{IV}(D)\ge \frac{1}{8}\hom(C_{4},D)$\\
(ii) & & $\hom_{II}(D)\le \frac{1}{2}\hom(C_{4},D)$

\end{tabular}
\end{proposition}

The proposition will follow easily from the following lemma that relates the counts of various types of oriented four-cycles in an oriented graph $D$.

\begin{lemma}\label{lem} The following inequalities hold for all oriented graphs $D$: 

\begin{tabular}{l p{0.1cm} l}

(i) & & $2\, \hom_{IV}(D)\ge \hom_{III}(D)$ \\
(ii) & & $\hom_{III}(D)\ge 2\, \hom_{I}(D)$ \\
(iii) & & $2\, \hom_{IV}(D)+\hom_{III}(D)\ge \hom_{II}(D)$\\
(iv) & & $\hom_{IV}(D)\ge  \hom_{I}(D)$ \\
(v) & & $4\, \hom_{IV}(D) \ge \hom_{II}(D)$\\ 

\end{tabular}
\end{lemma}

The results of the lemma are not deep, requiring only the fact that $(s-t)^{2}$ is non-negative for any $s,t\in\mathbb{R}$.  To reduce the proofs to such trivialities we must express the quantities $\hom_{\, \bullet}(D)$ in a suitable form.  

We begin with $\hom_{I}(D)$.  Recall that $\hom_{I}(D)$ counts the number of homomorphic copies in $D$ of the two oriented graph $\{\vec{12},\vec{23},\vec{34},\vec{41}\}$ and $\{\vec{14},\vec{43},\vec{32},\vec{21}\}$.  A homomorphic copy of $\{\vec{12},\vec{23},\vec{34},\vec{41}\}$ consists of $x_{1},x_{2},x_{3},x_{4}\in V(D)$ such that $\vec{x_{1}x_{2}},\vec{x_{2}x_{3}},\vec{x_{3}x_{4}},\vec{x_{4}x_{1}}\in E(D)$.  Thus, for each pair $x_{1},x_{3}\in V(D)$ the number of homomorphic copies in which they play the roles of first and third vertex is exactly $d^{+-}(x_{1},x_{3})d^{-+}(x_{1},x_{3})$.  A similar argument gives that $\hom(\{\vec{14},\vec{43},\vec{32},\vec{21}\},D)=\sum_{x_{1},x_{3}\in V}d^{-+}(x_{1},x_{3})d^{+-}(x_{1},x_{3})$.  Renaming $x_{1}$ as $x$ and $x_{3}$ as $x'$, we obtain \begin{equation}\label{homI} \hom_{I}(D)=2\sum_{x,x'\in V}d^{+-}(x,x')d^{-+}(x,x') \end{equation}   By following similar reasoning \vspace{-0.3cm}

\begin{eqnarray} \label{homII}
\hom_{II}(D)& = & 2 \sum_{x,x'\in V}\big( d^{++}(x,x')+d^{--}(x,x')\big ) \big( d^{+-}(x,x')+d^{-+}(x,x') \big) \vspace{0.3cm}\qquad \\ \label{homIII1}
\hom_{III}(D) & = & 4\sum_{x,x'\in V}d^{++}(x,x')d^{--}(x,x') \vspace{0.1cm} \\ \label{homIII2} & = & 2 \sum_{x,x'\in V}d^{+-}(x,x')^{2}+d^{-+}(x,x')^{2} \vspace{0.3cm} \\ \label{homIV}
\hom_{IV}(D) & = & \sum_{x,x'\in V}d^{++}(x,x')^{2}+d^{--}(x,x')^{2}  \vspace{0.3cm} 
\end{eqnarray}

The proof of the lemma is now trivial using the inequality $(s-t)^{2}\ge 0$.  In the proof we simply mention which expressions one is required to use.

\begin{proof}[Proof of Lemma \ref{lem}] (i) Use (\ref{homIII1}) and (\ref{homIV}).

(ii) Use (\ref{homI}) and (\ref{homIII2}).

(iii)  Use (\ref{homII}), (\ref{homIII2}) and (\ref{homIV}).

(iv) follows from (i) and (ii).

(v) follows from (i) and (iii).
\end{proof}

\begin{proof}[Proof of Proposition~\ref{propIIIV}] (i) From parts (i), (iv) and (v) of Lemma \ref{lem} we have that $7\hom_{IV}(D)\ge \hom_{I}(D)+\hom_{II}(D)+\hom_{III}(D)$.  Combining this with the expression (\ref{homup}) for $\hom(C_{4},D)$ we deduce that $8\hom_{IV}\ge \hom(C_{4},D)$, as required.

(ii) Writing $a(x,x')$ for $d^{++}(x,x')+d^{--}(x,x')$ and $b(x,x')$ for $d^{+-}(x,x')+d^{-+}(x,x')$, our expression for $\hom_{II}(D)$ becomes \begin{equation*} \hom_{II}(D)=2 \sum_{x,x'\in V}a(x,x')b(x,x') \end{equation*} While $\hom(C_{4},D)$, the number of homomorphic copies of all four-cycles whatever their orientation, may be expressed as \begin{equation*} \hom(C_{4},D)= \sum_{x,x'\in V}\big(a(x,x')+b(x,x') \big)^{2} \end{equation*} It is now easily observed that $\hom(C_{4},D)-2\hom_{II}(D)\ge 0$. \end{proof} 

The fact that Type IV cycles are never under-represented, and Type II cycles are never over-represented is strongly linked to their utility in quasi-randomness conditions.  Indeed, the fact that Type I and Type III cycles are over-represented in some oriented graphs (eg. in blow ups of a Type 1 cycle) and under-represented in others (eg. in blow ups of an arc), makes counts of these types of cycle unsuitable as quasi-randomness conditions.

\section{Proof of Theorem~\ref{equiv}}\label{secproof}

We write $P i \Rightarrow P j$ if $P i$ implies $P_j$ in the sense required by Theorem~\ref{equiv}.  We prove Theorem~\ref{equiv} by proving the following implications:

\vspace{0.3cm}\hspace{3cm}\xymatrix{ \textup{P1} \ar@{=>}[r] &  \textup{P2} \ar@{<=>}[r] & \textup{P5} \ar@{<=>}[r] & \textup{P8} \ar@{<=>}[r] & \textup{P9} \\ \textup{P3} \ar@{=>}[u] & \textup{P4} \ar@{=>}[l] &  \textup{P6} \ar@{=>}[l] \ar@{<=>}[r] \ar@{<=}[u] & \textup{P7}  } \vspace{0.6cm}

We remark that P2, P5 and P8 are equivalent in the stronger sense that they count exactly the same thing.  So the proofs that P2 $\Leftrightarrow$ P5 and P5 $\Leftrightarrow$ P8 simply rely on observing this fact.

Perhaps the most interesting aspect of Theorem~\ref{equiv} is that properties such as P1 and P2 (close to correct homomorphism counts of certain orientation of a four cycle) imply properties such as P3 and P4 (close to correct homomorphism counts of all small oriented (or partially oriented) graphs).  Since P4 includes P3, P2 is easily deduced from P1 and P2 is equivalent to P5, the real work is in proving the implications P5 $\Rightarrow$ P6 and P6 $\Rightarrow$ P4.  The first implication (P5 $\Rightarrow$ P6) relies on two applications of the Cauchy-Schwarz inequality.  The second implication (P6 $\Rightarrow$ P4), which states that if between any two large sets $A,B\subset V(D)$ the number of arcs in each direction is roughly equal then $D$ contains a close to correct homomorphism count of all small partially oriented graphs $H$, is proved inductively on the number of edges of $H$ that are oriented.  The base case, where no edges are oriented, is trivial, while the induction step relies on showing that approximately half of all homomorphic copies of $H'$ (which is defined by forgetting the orientation of one oriented edge of $H$) are in fact also homomorphic copies of $H$.

\subsection*{P1 $\Rightarrow$ P2}

It follows easily from Proposition~\ref{propIIIV} and parts (i) and (iv) of Lemma~\ref{lem} that P1 with parameter $\alpha$ implies $P2$ with $\delta=8\alpha$.

\subsection*{P2 $\Leftrightarrow$ P5}

Define the set $C\subset V^{4}$ as follows:
$$ C=\{(x,x',y,y'):|M_{xy}M_{xy'}M_{x'y}M_{x'y'}|=1\}\, .$$
Define a partition $C=C^{+}\cup C^{-}$ by including in $C^{+}$ the quadruples for which $M_{xy}M_{xy'}M_{x'y}M_{x'y'}$ takes the value $+1$.  It is easily verified that 
\begin{eqnarray*}  &\hom(C_{4},D)& =|C|=|C^{+}|+|C^{-}|\\
& \hom_{II}(D) & =|C^{-}|\\
\text{and}\quad & \sum_{x,x',y,y'}M_{xy}M_{xy'}M_{x'y}M_{x'y'}& =|C^{+}|-|C^{-}|\, .\end{eqnarray*}
The equivalence of the conditions P2 and P5 now follows immediately from the observation that  
\begin{equation*} \hom(C_{4},D)-2\hom_{II}(D)= \sum_{x,x',y,y'}M_{xy}M_{xy'}M_{x'y}M_{x'y'} \, .\end{equation*}

\subsection*{P5 $\Leftrightarrow$ P8} 

The sum of the fourth powers of the eigenvalues of a matrix is equal to the trace of the fourth power of that matrix.  Thus,
\begin{eqnarray*}  \sum_{i=1}^{n}\la_{i}^{4} & = & trace(M^{4})\\ 
& = &\sum_{x,x',y,y'\in V} M_{xy}M_{yx'}M_{x'y'}M_{y'x} \\
& = & \sum_{x,x',y,y'\in V} M_{xy}M_{xy'}M_{x'y}M_{x'y'}  \, .\end{eqnarray*} 

\subsection*{P8$\Leftrightarrow$ P9}

It is trivial to deduce P9 from P8.  Indeed, if we assume P8($\delta$), then $\la_{1}^{4}\le \sum_{i=1}^{n}\la_{i}^{4}\le \delta n^{4}$ and so $\la_{1}\le \delta^{1/4} n$.  

On the other hand, if $\la_{1}\le \zeta n$ then, using the fact that the eigenvalues $\lambda_i$ are imaginary, we obtain the bound 
$$\sum_{i=1}^{n}|\lambda_{i}^{2}|=-\sum_{i=1}^{n}\lambda_{i}^{2}=-trace{M^{2}}=2e(D)\le n^{2}\, .$$
Thus
$$\sum_{i=1}^{n}\lambda_{i}^{4}\le |\lambda_{1}^{2}|\sum_{i=1}^{n}|\lambda_{i}^{2}|\le \zeta^{2}n^{4}\, .$$

\subsection*{P5 $\Rightarrow$ P6} 

We assume P5($\delta$).  Let $A,B$ be subsets of $V$ and note that 
$$\vec{e}(A,B)-\vec{e}(B,A)=\sum_{x\in A}\sum_{y\in B}M_{xy}\, .$$  
We bound this quantity using two applications of the Cauchy-Schwarz inequality, for this reason we shall in fact bound $(\vec{e}(A,B)-\vec{e}(B,A))^{4}$.  We write \emph{by C-S} to mark each application of the Cauchy-Scwartz inequality, on these occasions we also use that $|A|,|B|\le n$.

\begin{displaymath} 
\begin{array}{r l l r} 
\big(\vec{e}(A,B)-\vec{e}(B,A)\big)^{4} = & \Bigg(\Big(\sum_{x\in A}\sum_{y\in B}M_{xy}\Big)^{2}\Bigg)^{2} & & \vspace{0.4cm} \\
\le & \Bigg(n \sum_{x\in A}\Big(\sum_{y\in B}M_{xy}\Big)^{2}\Bigg)^{2} & \phantom{phantom}& \text{by C-S}  \vspace{0.4cm} \\
\le & n^{2} \Bigg( \sum_{x\in V} \Big(\sum_{y\in B}M_{xy}\Big)^{2}\Bigg)^{2} & & \vspace{0.4cm} \\
= & n^{2} \Bigg( \sum_{x\in V}\sum_{y,y'\in B}M_{xy}M_{xy'}\Bigg)^{2} & & \vspace{0.4cm} \\
= & n^{2} \Bigg(\sum_{y,y'\in B}\sum_{x\in V}M_{xy}M_{xy'}\Bigg)^{2} & & \vspace{0.4cm} \\
\le & n^{2} n^{2} \sum_{y,y'\in B}\Big(\sum_{x\in V}M_{xy}M_{xy'}\Big)^{2} & &\text{by C-S} \vspace{0.4cm} \\
\le & n^{4} \sum_{y,y'\in V}\Big(\sum_{x\in V}M_{xy}M_{xy'}\Big)^{2} & & \vspace{0.4cm} \\
= & n^{4} \sum_{y,y'\in V}\sum_{x,x'\in V} M_{xy}M_{xy'}M_{x'y}M_{x'y'} & & \vspace{0.4cm} \\
\le & \delta n^{8} \, .
\end{array}
\end{displaymath}
Thus $\vec{e}(A,B)-\vec{e}(B,A)\le \delta^{1/4}n^{2}$.  Hence P5($\delta$) implies P6($\gamma$) with $\gamma=\delta^{1/4}$.

\subsection*{P6 $\Rightarrow$ P4}  Our proof is by induction, for this reason we must state formally the result we shall prove.  Note that the proposition does indeed prove that P6($\gamma$) implies P4($\eta$) for $\eta=\gamma$.

\begin{proposition} Let $D$ be an oriented graph on $n$ vertices satisfying P6($\gamma$).  Let $H$ be a partially oriented graph on $k$ vertices.  Then \begin{equation*}\Bigg | hom(H,D)- \frac{hom(\bar{H},D)}{2^{\vec{e}(H)}}\Bigg | \le (1-2^{-\vec{e}(H)})\gamma n^{k}\, .\end{equation*}\end{proposition}

We shall use in the proof the following, equivalent, form of P6($\gamma$):
\begin{equation}\label{P6equiv} \vec{e}(B,A)\ge \frac{e(A,B)}{2} -\frac{\gamma}{2}n^{2}\qquad \text{for all}\, A,B\subset V\, .\end{equation}
   
\begin{proof} We prove the proposition by induction on $\vec{e}(H)$.  If $\vec{e}(H)=0$, then $H=\bar{H}$, and so $hom(H,D)=\hom(\bar{H},D)$.  For the general case, let $H$ be an oriented graph on $\{1,...,k\}$ with $\vec{e}(H)\ge 1$.  By relabelling if necessary (which does not affect the homomorphism count) we may assume that $\vec{12}$ is an arc of $H$.  Let $H'$ be the partially oriented graph obtained by unorienting this edge.  We now relate the quantities $\hom(H',D)$ and $\hom(H,D)$.  For each $(x_{3},...,x_{k})\in V^{k-2}$ let $\Hom(H',D;\, .\, ,\, .\, ,x_{3},...,x_{k})$ denote the set of homomorphisms $\phi$ of $H'$ into $D$ for which $\phi(i)=x_{i}$ for all $i=3,...,k$.  Similarly define $\Hom(H,D;\, .\, ,\, .\, ,x_{3},...,x_{k})$.  In fact, it is easy to characterise the homomorphisms $\phi \in \Hom(H',D;\, .\, ,\, .\, ,x_{3},...,x_{k})$.  A homomorphism $\phi \in \Hom(H',D;\, .\, ,\, .\, ,x_{3},...,x_{k})$ must have $\phi(i)=x_{i}$ for $i=3,...,k$, and must pick values for $\phi(1)$ and $\phi(2)$.  Writing $x_{1}$ for $\phi(1)$ we know $x_{1}$ must join up appropriately to the vertices $x_{3},...,x_{k}$.  Specifically

\begin{tabular}{c p{1cm} p{10cm}} 

(i) & & $\vec{x_{1}x_{i}}$ is an arc of $D$, for every arc $\vec{1i}:i\ge 3$ in $H$.
\\
(ii) & & $\vec{x_{i}x_{1}}$ is an arc of $D$, for every arc $\vec{i1}:i\ge 3$ in $H$.
\\
(iii) & & $x_{1}x_{i}$ is an edge of $\bar{D}$, for every edge $1i:i\ge 3$ in $\bar{H}$.

\end{tabular}

Equivalently, we require $x_{1}\in \bigcap_{i\ge 3:\vec{1i}\in E(H)}\Gamma^{-}(x_{i})\cap \bigcap_{i\ge 3:\vec{i1}\in E(H)}\Gamma^{+}(x_{i})\cap \bigcap_{i\ge 3:1i\in E(\bar{H})}\Gamma(x_{i})$.  We denote this set $A$.  Similarly, writing $x_{2}$ for $\phi(2)$, there are similar restrictions on $x_{2}$, which again are equivalent to demanding that $x_{2}$ belongs to a certain set, we denote this set $B$.  Since $H'$ has an unoriented edge between $1$ and $2$, we have a final condition - the condition that $x_{1}x_{2}$ is an edge of $\bar{D}$.  Hence for certain sets $A$ and $B$, we have a one-to-one correspondence between homomorphisms $\phi\in \Hom(H',D;\, .\, ,\, .\, ,x_{3},...,x_{k})$ and edges of $\bar{D}$ between $A$ and $B$.

Similarly, we may characterise the homomorphisms $\phi\in \Hom(H,D;\, .\, ,\, .\, ,x_{3},...,x_{k})$.  Again we write $x_{1}$ and $x_{2}$ for $\phi(1)$ and $\phi(2)$. The restrictions $x_{1}\in A$ and $x_{2}\in B$ remain.  However, on this occasion we require not only that there is some edge between $x_{1}$ and $x_{2}$, we require that there is an oriented edge \emph{from} $x_{1}$ \emph{to} $x_{2}$.  Thus, there is a one-to-one correspondence between homomorphisms $\phi\in \Hom(H,D;\, .\, ,\, .\, ,x_{3},...,x_{k})$ and edges from $A$ to $B$.

Thus $|\Hom(H',D;\, .\, ,\, .\, ,x_{3},...,x_{k})|$ and $|\Hom(H,D;\, .\, ,\, .\, ,x_{3},...,x_{k})|$ are $e(A,B)$ and $\vec{e}(A,B)$ respectively, for some pair of subsets $A,B\subset V$.  From our condition (\ref{P6equiv}) we obtain that \begin{equation*} \Bigg | |Hom(H,D;\, .\, ,\, .\, ,x_{3},...,x_{k})|- \frac{|Hom(H',D;\, .\, ,\, .\, ,x_{3},...,x_{k})|}{2}\Bigg |\le \frac{\gamma}{2}n^{2}\, .\end{equation*} Since $\hom(H,D)$ is the sum over $(x_{3},...,x_{k})\in V^{k-2}$ of $|\Hom(H,D;\, .\, ,\, .\, ,x_{3},...,x_{k})|$, and similarly $\hom(H',D)$, we have that
\begin{equation*} \Bigg | \hom(H,D)-  \frac{\hom(H',D)}{2}\Bigg | \le \frac{\gamma}{2}n^{k}\, .\end{equation*} 
Having obtained this relation between $\hom(H,D)$ and $\hom(H',D)$ we require only an application of the induction hypothesis.  As $\vec{e}(H')=\vec{e}(H)-1$, an application of the induction hypothesis to $H'$ gives $|\hom(H',D)- \hom(\bar{H},D)/2^{\vec{e}(H)-1}|\le (1-2^{1-\vec{e}(H)})\gamma n^{k}$.  Combining this with the inequality proved above
\begin{equation*} \Bigg | \hom(H,D)-  \frac{\hom(\bar{H},D)}{2^{\vec{e}(H)-1}}\Bigg |\le \frac{\gamma}{2}n^{k}+\frac{1}{2} (1-2^{1-\vec{e}(H)})\gamma n^{k}= (1-2^{-\vec{e}(H)})\gamma n^{k}\, .\end{equation*}
\end{proof}

\subsection*{P4 $\Rightarrow$ P3 $\Rightarrow$ P1}

It is trivial that P4 implies P3.  While P3 with parameter $\beta$ implies P1 with parameter $\alpha=2\beta$, as $\hom_{IV}(D)$ is simply the sum of the number of homomorphic copies of two particular orientations of the four-cycle.

\subsection*{P6 $\Leftrightarrow$ P7} 

Suppose that P6($\gamma$) holds and let $A,B$ be a pair with $\vec{e}(B,A)\le (1-\ep) \vec{e}(A,B)$.  Then 
\begin{equation*}  \ep \vec{e}(A,B)\le \vec{e}(A,B) - \vec{e}(B,A) \le \gamma n^{2}\end{equation*}
and so $\vec{e}(A,B)\le \ep^{-1}\gamma n^{2}$.  Thus P6($\gamma$) implies P7($\ep$) where $\ep=\gamma^{1/2}$.  

On the other hand, if we assume P7($\ep$) then $\vec{e}(A,B)-\vec{e}(B,A)\le \vec{e}(A,B)\le \ep n^{2}$ for all pair $A,B$ with $\vec{e}(B,A)\le (1-\ep)n^{2}$ and $\vec{e}(A,B)-\vec{e}(B,A)\le \ep \vec{e}(A,B)\le \ep n^{2}$ for all other pairs.  Thus P7($\ep$) implies P6($\gamma$) where $\gamma=\ep$.


\begin{thebibliography}{}

\bibitem{AGH} O. Amini, S. Griffiths and F. Huc, Subgraphs of weakly quasi-random oriented graphs, \emph{SIAM Journal of Discrete Mathematics}, \textbf{25} (2011), 234 -- 259. 

\bibitem{CG3} F. R. K. Chung and R. L. Graham, Quasi-random classes of hypergraphs, \emph{Random Structures and Algorithms} \textbf{1} no. 4 (1990), 363 -- 382.

\bibitem{CG} F. R. K. Chung and R. L. Graham, Quasi-random tournaments, \emph{Journal of Graph Theory} \textbf{15} (1991), 173 -- 198.

\bibitem{CG2} F. R. K. Chung and R. L. Graham, Quasi-random set systems, \emph{Journal of the AMS} \textbf{4} (1991), 151--196.

\bibitem{CGW} F. R. K. Chung, R. L. Graham and R. M. Wilson, Quasi-random graphs, \emph{Combinatorica} \textbf{9} 
(1989), 345 -- 362. 

\bibitem{G} W. T. Gowers, Quasirandomness, Counting and Regularity for $3$-Uniform Hypergraphs, \emph{Combinatorics, Probability and Computing} \textbf{15} no. 1-2 (2006), 143 -- 184.

\bibitem{HL} H. Huang and C. Lee, Quasi-randomness of graph balanced cut properties, \emph{Random Structures Algorithms}, to appear.

\bibitem{KS} S. Kalyanasundaram and A. Shapira, A Note on Even Cycles and Quasi-Random Tournaments, arXiv:1108.0011 [math.CO]  

\bibitem{KS} M. Krivelevich and B. Sudakov, Pseudo-random graphs, \emph{More Sets, Graphs and Numbers, Bolyai Soc. Math. Stud} \textbf{15} Springer-Verlag (2006), 199 -- 262.

\bibitem{S} A. Shapira, Quasi-randomnesss and the distribution of copies of a fixed graph, \emph{Combinatorica} \textbf{28}  no.6 (2008), 735 -- 745.

\bibitem{SS} M. Simonovits and V. T. S\' os, Hereditarily extended properties, quasi-random graphs and not necessarily induced subgraphs, \emph{Combinatorica} \textbf{17} (1997), 577 -- 596.

\bibitem{SY} A. Shapira and R. Yuster, The quasi-randomness of hypergraph cut properties, \emph{Random Structures Algorithms}, to appear.

\bibitem{T1} A. Thomason, Pseudo-random graphs, Proc. of Random Graphs, Pozna\' n 1985, M. Karo\' nski, 
ed., \emph{Annals of Discrete Math.} \textbf{33} (North Holland 1987), 307 -- 331. 

\bibitem{T2} A. Thomason, Random graphs, strongly regular graphs and pseudo-random graphs, Surveys in 
Combinatorics, C. Whitehead, ed., \emph{LMS Lecture Note Series} \textbf{123} (1987), 173 -- 195.

\bibitem{Y} R. Yuster, Quasi-randomness is determined by the distribution of copies of a fixed graph in equicardinal large sets, \emph{Combinatorica} \textbf{30} no. 2 (2010), 239 -- 246.

\end{thebibliography}
\end{document}